\documentclass[10pt]{article}
\usepackage{amssymb, amscd, amsthm, amsmath,latexsym, amstext,mathrsfs,verbatim}
\usepackage[all]{xy}
\usepackage{graphicx,enumerate}
\usepackage{url}

\def\ddD{{\rm D}}

\def\ddB{{\rm B}}

\def\ddD{{\rm D}}

\newcommand{\cA}{\mathcal{A}}

\newcommand{\cH}{\mathcal{H}}

\newcommand{\cO}{\mathcal{O}}

\newcommand{\C}{\mathbb C}

\newcommand{\N}{\mathbb N}

\newcommand{\R}{\mathbb R}
\newcommand{\Z}{\mathbb Z}

\def\mod{\mathrm{mod}}

\DeclareMathOperator{\GL}{\mathrm GL}

\numberwithin{equation}{section}

\newtheorem{lemma}{Lemma}[section]
\newtheorem{cor}[lemma]{Corollary}

\newtheorem{thm}[lemma]{Theorem}

\theoremstyle{definition}

\theoremstyle{remark}
\newtheorem{rem}[lemma]{Remark}

\begin{document}
\title{Clifford theory on rational Cherednik algebras of imprimitive groups }
\author{Shoumin Liu}
\date{}

\maketitle
%\mainmatter
%\setcounter{section}{-1} \tableofcontents
%\tableofcontents
\begin{abstract}
Ram and Rammage have introduced an automorphism and Clifford theory on affine Hecke algebras. Here we will extend them to cyclotomic  Hecke algebras
and rational Cherednik algebras.
\end{abstract}

\section{Introduction}
Symmetric groups are  important classical groups in representation theory.
Usually, we consider the symmetric group ${\bf S}_n$ as a subgroup in ${\GL}(n,\C_n)$ as those matrices with just $n$ $1$s and $n(n-1)$ $0$s as their  matrix items. If we extend it to those matrices in ${\GL}(n,\C_n)$ with just $n$ nonzero items but those nonzero items  have values of $r$th root of $1$, then we obtain the imprimitive group $G(r,1,n)$. There exists a subgroup  the subgroup $G(r,p,n)$ of $G(r,1,n)$, for those having determinant of $\frac{r}{p}$th root of $1$ for $p\mid r$.\\

In \cite{AK1994} and \cite{A1995}, the representations  of Hecke algebras associated to $G(r,1,n)$
and $G(r,p,n)$ are given; furthermore, In \cite{A1995}, the author gives a shift operator to obtain the
irreducible representation of type $G(r,p,n)$ from type $G(r,1,n)$, which can be considered as an
application of Clifford theory.
In \cite{RR2003}, the Clifford theory on affine Hecke algebras are studied with combinatorial method applied.
Recently, there are some conclusions of rational Cherednik algebras which are closely related to Hecke algebras, especially in \cite{GGOR2003} and \cite{BC2011}. In \cite{S2011}, some connection between the rational Cherednik algebras of type
$G(r,1,n)$ and crystals of Fock spaces are proved. There is a Knizhink-Zamolodchikov functor which builds a bridge between the representations of rational Cherenik algebras and Hecke algebras of the corresponding types. \\
In this paper, definitions and basic properties of the related groups, algebras will be introduced. We will show the analogous conclusion as in \cite{RR2003} for the cyclotomic  Hecke algebra and prove the clifford operator are the inverse of Ariki's shift operator on the irreducible representations of cyclotomic Hecke algebras.  We also extend the clifford operators to the rational Cherednik algebra of type $G(r,1,n)$, and prove that the subalgebra invariant under Clifford operator is the rational Cherednik algebra of type
$G(r,p,n)$. The commutating relations of restriction functor and Knizhnik-Zamolodchikov functor, and the Clifford functor and  Knizhnik-Zamolodchikov functor are proved in this paper.\\
We sketch the paper in the following.  First we recall the  definitions and basic properties about rational Cherednik algebras and cyclotomic  Hecke algebras in Section \ref{sect:defn}. In Section \ref{sect:auto}, we extend the automorphism in \cite{RR2003} to these algebras. Then we present the relation of the automorphism on modules and the shift operator given by Ariki in Section \ref{sect:shift}. In  Section \ref{sect:res and KZ},
we introduce the ${\rm KZ}$ functor and prove the commutating relations of restriction functor and ${\rm KZ}$ functor, and the Clifford functor and  ${\rm KZ}$ functor, then we end this by the classical example of type $\ddB_n$ and type $\ddD_n$.

\section{Rational Cherednik algebras of imprimitive groups}\label{sect:defn}
%Let $G$ be a semisimple algebraic group over field $\mathbb{F}_q$,
%$B$ a Borel subgroup of $G$, $T$ be a maximal Torus in $B$, $\theta\in \hat{T}$ a character of $T$.
%$${\rm Ind}_{B}^{G}(\theta)=\{f:G\mapsto \C\mid f(bg)=\theta(b)f(g), \forall b\in B,\,g\in G.\}.$$
\subsection{Rational Cherednik algebras}
Let $V$ be an $n$-dimensional Hilbert space and $U(V)$ be the group of
unitary transformation on $V$. Suppose that $W\subset U(V)$
is  a finite \emph{complex reflection group}.
Let $\mathcal{C}$ be the set of $W$-orbits of reflection hyperplanes and let $\mathcal{A}$ be the full hyperplane arrangement,namely $\cup_{C\in \mathcal{C}}C$.
For a hyperplane $H$ in an orbit $C\in \mathcal{C}$, the stabilizer
$$W_{H}: =\{w\in W: wx=x \quad for\, any \quad x\in H \}$$
is isomorphic to a cyclic group $\Z/e_C\Z$. Here, $e_C$ is the order of $W_H$, which is determined
by the orbit $C$.\\
Let $\gamma:\cA\rightarrow \C W$, $H\mapsto \gamma_H$ be a $W$-equivariant map such
that $\gamma_{H}$ is an element of $\C W_H\subset\C W$ with trace zero. Giving $\gamma$ is equivalent to
giving an $W$-invariant function $c$ on all reflections  such that
\begin{eqnarray}\label{gamma}
\gamma_H=\sum_{g\in W_H\setminus\{1\}}c_g\cdot g.
\end{eqnarray}
The Cherednik algebras $H_c(W)$ can  be defined directly, in terms of generators and relations.
The algebra $H_c(W)$ is generated by elements $x\in V^{*}$, $y\in V$ and $w\in W$ subject to
the following relations:
\begin{eqnarray}\label{daharel1}
[x, x']=0,\quad [y,y']=0, \quad wxw^{-1}=w(x), \quad wy w^{-1}=w(y)
%\\
%[\xi, x]&=&<\xi, x>+
%\sum_{H\in \mathcal{A}}\frac{\left<\alpha_H, \xi\right> \left<x, v_H\right>}{<\alpha_H, v_H>}\sum_{i=0}^{e_H-1}e_H(K_{H,i}-K_{H,i+1})e_{H,i}
\end{eqnarray}
\begin{eqnarray}\label{daharel2}
[y, x]=\left<y, x\right>+
\sum_{H\in \mathcal{A}}\frac{\left<\alpha_H, y\right> \left<x, v_H\right>}{\left<\alpha_H, v_H\right>}\gamma_H
\end{eqnarray}
where $\alpha_H\in V^{*}$(the dual space of $V$ under a $W$-invariant inner product) is the linear equation defining $H$; we take $\alpha_H=v^{*}_H$, where $v_H$ is a norm of $H$.\\
\subsection{Imprimitive groups}
The imprimitive group $G(r,1, n)$ are defined
with generators $\{S_i\}_{i=1}^{n}$ and the following relations,
\begin{enumerate}[(1)]
\item $S_i S_j=S_j S_i$, if $|i-j|>1$,
\item $S_i S_{i+1}S_i= S_{i+1}S_i S_{i+1}$, for $2\leq i\leq n-1$,
\item $S_1 S_2 S_1S_2=S_2 S_1S_2S_1$,
\item $S_1^r=S_i^2=1$, for $2\leq i\leq n$.
\end{enumerate}
The group $G(r,p,n)$ with $p\mid r$ is the subgroup  of index $p$ in $G(r,1, n)$ generated by
\begin{eqnarray*}
s_0=S_1^p, \quad s_1=S_1S_2S_1, \quad s_i=S_i, \quad 2\leq i\leq n.
\end{eqnarray*}
Let $u_1$, $\ldots$, $u_r$, $0\neq q\in\C$. The cyclotomic Hecke algebra $\cH_{r,1,n}(u_1, \ldots, u_r, q)$
is the algebra generated over $\C$ given by generators $\{T_i\}_{i=1}^{n}$ and the following relations,
\begin{enumerate}[(1)]
\item $T_i T_j=T_j T_i$, if $|i-j|>1$,
\item $T_i T_{i+1}T_i= T_{i+1}T_i T_{i+1}$, for $2\leq i\leq n-1$,
\item $T_1 T_2 T_1T_2=T_2 T_1T_2T_1$,
\item $(T_1-u_1)(T_1-u_2)\cdots(T_1-u_r)=0$,
\item    $(T_i-1)(T_i+q)=0$,  for $2\leq i\leq n$.
\end{enumerate}
The algebra $\cH_{r,1,n}(u_1, \ldots, u_r, q)$ is of dimension $r^n n!$ (see \cite{AK1994}).
Let $d=r/p$. Let $\{v_i\}_{i=0}^{d-1}\subset \C$ and $\xi$ be a primitive $p$th root of $1$. for
$1\leq j\leq r$, define
\begin{eqnarray}\label{uvrel}
u_j=\xi^k v_l \quad {\rm if} \, j-1=lp +k, \quad (0\leq k\leq p-1,\, 0\leq l\leq d-1)
\end{eqnarray}
i.e., $\{u_i\}_{i=1}^{r}$ are chosen so that
\begin{eqnarray}\label{t1uvrel}
(T_1-u_1)(T_1-u_2)\cdots(T_1-u_r)=(T_1^p-v_0^p)(T_1^p-v_1^p)\cdots(T_1^p-v_{d-1}^{p})
\end{eqnarray}
The algebra $\cH_{r,p,n}(v_0,\ldots, v_d,q)$ is the subalgebra of $\cH_{r,1,n}(u_1, \ldots, u_r, q)$
generated by the elements
\begin{eqnarray*}
t_0=T_1^p, \quad t_1=T_1^{-1}T_2T_1, \quad t_i=T_i, \quad 2\leq i\leq n.
\end{eqnarray*}
If $\zeta$ is a primitive $r$th root of $1$,
then we have
\begin{eqnarray*}
\cH_{r,1,n}(1,\zeta, \ldots, \zeta^{r-1}, q)\cong \C G(r,1,n).
\end{eqnarray*}
If $\eta$ is a primitive $d$th root of $1$,
then we have
\begin{eqnarray*}
\cH_{r,p,n}(1,\eta, \ldots, \eta^{d-1}, q)\cong \C G(r,p,n).
\end{eqnarray*}
Without the defining relation (4) for $\cH_{r,1,n}(u_1, \ldots, u_r, q)$, the affine Hecke
algebra  is obtained, which we denote  as $\cH_{\infty,1,n}(q)$. To distinguish $\cH_{\infty,1,n}(q)$ from  $\cH_{r,1,n}(u_1, \ldots, u_r, q)$, we replace the generator $T_1$ by $X$, therefore
the generators of  $\cH_{\infty,1,n}(q)$ are $\{X_1\}\cup\{T_i\}_{i=2}^{n}$.
Just as $\cH_{r,p,n}(v_0,\ldots, v_d,q)$ rising from  $\cH_{r,1,n}(u_1, \ldots, u_r, q)$,
we define $\cH_{\infty, p,n}(q)$ is the subalgebra of $\cH_{\infty,1,n}(q)$  with generators
\begin{eqnarray*}
t_0=X^p, \quad t_1=X^{-1}T_2X, \quad t_i=T_i, \quad 2\leq i\leq n.
\end{eqnarray*}
Naturally we have the following commutative diagram, where we denote the quotient map from
$\cH_{\infty,1,n}(q)$ to $\cH_{r,1,n}(u_1, \ldots, u_r, q)$ by $\phi$.\\
\begin{center}
\xymatrix{
\cH_{\infty, p,n}(q)\quad\ar@{^{(}->}[r]\ar@{->>}[d]^{\phi}&\quad\cH_{\infty,1,n}(q)\ar@{->>}[d]^{\phi}\\
\cH_{r,p,n}(v_0,\ldots, v_d,q)\ar@{^{(}->}[r]&\cH_{r,1,n}(u_1, \ldots, u_r, q)
}
\end{center}
\section{Automorphism on $H_c(G(r,1,n))$}\label{sect:auto}
Let's first recall Theorem {\bf 2.2} in \cite{RR2003}.
\begin{thm}Let $\xi$ be a primitive $p$th root of unity.
The algebra automorphism $\tau : \cH_{\infty,1,n}(q)\rightarrow \cH_{\infty,1,n}(q)$ defined by
\begin{eqnarray*}
\tau(X)=\xi X,  \quad \tau(T_i)=T_i, \quad 2\leq i\leq n.
\end{eqnarray*}
gives rise to an action of the group $\Z/p\Z=\left< \tau\right>$ on $\cH_{\infty,1,n}(q)$ by algebra
automorphism and
$$\cH_{\infty,p,n}=(\cH_{\infty,1,n})^{\Z/p\Z}.$$
\end{thm}
We apply the argument analogous to \cite{RR2003} to prove the following theorem.
\begin{thm} \label{tauonha}
Let $\xi$ be a primitive $p$th root of unity and $ \cH_{r,1,n}(u_1, \ldots, u_r, q)$  be as defined above with
parameters satisfying the relations (\ref{uvrel}) and (\ref{t1uvrel}).
the map $\tau$ defined on the generators to $ \cH_{r,1,n}(u_1, \ldots, u_r, q)$ as the below,
\begin{eqnarray*}
\tau(T_1)=\xi T_1,  \quad \tau(T_i)=T_i, \quad 2\leq i\leq n.
\end{eqnarray*}
 induces an algebra isomorphism on $ \cH_{r,1,n}(u_1, \ldots, u_r, q)$
and an action of the group $\Z/p\Z=\left< \tau\right>$ on $\cH_{\infty,1,n}(q)$  and
$$\cH_{r,p,n}(v_0,\ldots, v_d,q)=(\cH_{r,1,n}(u_1, \ldots, u_r, q))^{\Z/p\Z}.$$
Furthermore, we have the following commutative diagram, \\
 \xymatrix{
\cH_{\infty, p,n}(q)\quad\ar[r]\ar[d]^{\phi}&\cH_{\infty, p,n}(q)\quad\ar[r]^{\tau}\ar[d]^{\phi}&\quad\cH_{\infty,1,n}(q)\ar[d]^{\phi}\\
\cH_{r,p,n}(v_0,\ldots, v_d,q)\ar[r]&\cH_{r,1,n}(u_1, \ldots, u_r, q)\ar[r]^{\tau}&\cH_{r,1,n}(u_1, \ldots, u_r, q)
}
\end{thm}
\begin{proof}
First, one can verify that all relations in the definition of $\cH_{r,p,n}$ are kept under $\tau$ , especially the relation
(4) in the definition replaced by (\ref{t1uvrel}).\\
In $\cH_{\infty,1,n}(q)$, let $X^{\epsilon_1}=X$ and
$X^{\epsilon_{i}}=T_iX^{\epsilon_{i-1}}T_i$ for $2\leq i\leq n$.
View $\{\epsilon_i\}_{i=1}^{n}$ as a basis of $\R^{n}$. Let
$$L=\sum_{i=1}^{n}\Z\epsilon_i, \quad Q=\sum_{i=2}^{n}\Z(\epsilon_i-\epsilon_{i-1}),\quad L_p=Q+pL. $$
Let $X^{\lambda}=(X^{\epsilon_1})^{\lambda_1}\cdots (X^{\epsilon_n})^{\lambda_n}$
for $\lambda=\sum_{i=1}^{n}\lambda_i\epsilon_i\in L$.
The sets
$$\{X^{\lambda}T_w\mid \lambda\in L,w\in {\bf S}_n\}\quad and \{X^{\lambda}T_w\mid \lambda\in L_p,w\in {\bf S}_n\}$$
are bases of $\cH_{\infty,1,n}(q)$ and $\cH_{\infty, p,n}(q)$, respectively(\cite{L1989}).
Let $C_r=\{\lambda\in L\mid 0\leq \lambda_i<r\}$.
Ariki and Koile \cite{AK1994}, and Arike \cite{A1995} have shown that the sets
$$\{X^{\lambda}T_w\mid \lambda\in C_r,w\in {\bf S}_n\}\quad and\quad \{X^{\lambda}T_w\mid \lambda\in L_p\cap C_r,w\in {\bf S}_n\}$$
are bases of $\cH_{r,1,n}(u_1, \ldots, u_r, q)$ and $\cH_{r,p,n}(v_0,\ldots, v_d,q)$,
 respectively. \\
   Take an element $h\in \cH_{r,1,n}(u_1, \ldots, u_r, q)$, then $h$ can be
 written as $$h=\sum_{\lambda\in C_r, w\in S_{n}}c_{\lambda, w}\phi(X^{\lambda})w,$$ so
 $$\tau(h)=\sum_{\lambda\in C_r, w\in S_{n}}\xi^{|\lambda|}c_{\lambda, w}\phi(X^{\lambda})w,$$
 where $|\lambda|=\sum_{i=1}^{n}\lambda_{i}$. Hence  $\tau(h)=h$ if and only if $h\in \cH_{r,p,n}(v_0,\ldots, v_d,q)$.
\end{proof}
If we replace algebras $\cH_{r,p,n}(v_0,\ldots, v_d,q)$, $(\cH_{r,1,n}(u_1, \ldots, u_r, q))$  by
$\C[G(r,p,n)]$ and $\C[G(r,1,p)]$, respectively, and define the $\tau$-action on generators of
$\C[G(r,1,p)]$ similarly, we can obtain the following corollary.
\begin{cor}\label{tauongroup}
The map $\tau$ induces an algebra isomorphism on $\C[G(r,1,p)]$
and an action of the group $\Z/p\Z=\left< \tau\right>$ on $\C[G(r,1,p)]$.  We have
$$\C[G(r,p,n)]=(\C[G(r,1,p)])^{\Z/p\Z}.$$
\end{cor}
Here we want to extend the $\tau$ on to Rational Cherednik algebra associated to $G(r,1,n)$, but we need to give restriction on the $\gamma$-function on reflections. we require $c_g=0$ in (\ref{gamma}) if
$g\in G(r,1,n)\setminus G(r,p,n)$.
\begin{thm}\label{invDAHA}
Let $\tau$ act on $V$  as  the identity   and  on  $\C[G(r,1,p)]$
as in Corollary \ref{tauongroup}.
Then $\tau$ induces an algebra isomorphism on $H_c(G(r,1,p))$
and an action of the group $\Z/p\Z=\left< \tau\right>$ on $\C[G(r,1,p)]$.  We have
$$H_c(G(r,p,n))=(H_c(G(r,1,p)))^{\Z/p\Z}.$$
\end{thm}
\begin{proof} First, it can be verified easily that the relations (\ref{daharel1}) and (\ref{daharel2}) for defining $H_c(G(r,1,p))$ still hold as identities under the action of $\tau$, therefore it induces
an algebra isomorphism on  $H_c(G(r,1,p))$ with order $p$. The second claim follows from the same argument
as in Theorem \ref{tauonha} and the PBW basis given in \cite{EG2002}, where it was proved that
$H_c(G(r,1,p))\cong \C[V]\otimes C[V^{*}]\otimes \C[G(r,1,p)]$.

\end{proof}
\section{Ariki's shift operation and $\tau$}\label{sect:shift}

Let $R$ be an algebra and $G$ be a finite group acting by automorphism on $R$.\\
Let $N$ be a (finite dimensional) left $R$-module.  For each $g\in G$ define an $R$-module $^{g}N$,
which has the same underlying vector space $N$ but such that
$^{g}N$ has $R$-action given by
$$r\circ n=g^{-1}(r)n,$$
for $r\in R$, $n\in N$. Therefore $^{g}N$ is simple if and only if $N$ is simple. Thus there is an action
of $g$ on the set of simple $R$-module. \\
Let's describe the Ariki's shift operation in \cite{A1995}.\\
A Young diagram is a finite subset $\lambda$ of $\N\times \N$ which has the property that if
$(k,l)\in \lambda$, then $(k',l')\in \lambda$ for all $k'\leq k$, $l'\leq l$. The elements
of $\lambda$ are called cells of $\lambda$. We say that $(k,l)$ is the cell located in the
$k$-th row and the $l$-th column. \\
Let $\lambda=(\lambda_{j}^{i})_{0\leq i\leq p, 1\leq j\leq d}$ be an $r$-tuple of Young diagrams of size
$\sum_{i,j}|\lambda_{j}^{i}|=n$. A standard tableau $\mathbb{T}=(T_j^i)$ of shape $\lambda$ is , by definition, a bijection from $\{1,\dots,n\}$ to the set of cells of $\lambda$ such that if two
numbers $a<b$ are mapped to adjoint cells, $(k,l)$ and $(k',l')$, respectively, then  $k< k'$ or $l< l'$ holds. If $1\leq y\leq n$ is located in  $\mathbb{T}$, then we write $a(\mathbb{T},y)=i$ and $b(\mathbb{T}, y)=j$ if
$a$ is located in a cell of $\lambda^{i}_{j}$. The content of $y$ in $\mathbb{T}$ is defined by
$c(\mathbb{T},y)=l-k$ if $y$ is located in the $k$th row and $l$th column of $\lambda^{a(\mathbb{T},y)}_{b(\mathbb{T}, y)}$.\\
We define a shift operator on the set of $r$-tuple of Young diagrams or on the set of standard
tableaux by;
\begin{eqnarray}
\rm{shift}(\lambda)&=&\left(\lambda_j^{i+1}\right)\\
\rm{shift}(\mathbb{T})&=&\left(T_j^{i+1}\right)
\end{eqnarray}
for $\lambda=(\lambda_{j}^{i})$ and $\mathbb{T}=(T_j^i)$ respectively. Here we regard $i$ as $i$ modulo $p$.

\begin{thm} Suppose   that the algebra $\cH_{r,1,n}(u_1, \ldots, u_r, q)$ is semisimple and  satisfies the relation
(\ref{t1uvrel}) and $v_{i}^{p}\neq v_{j}^{p}$ for $i\neq j$. Up to some ordering of $u_i$, we have that
$$\tau \circ {\rm shift}=1$$
on the simple $\cH_{r,1,n}(u_1, \ldots, u_r, q)$-modules.
\end{thm}
\begin{proof}First, we need to recall the defining action on the $r$-tuple Young diagram from \cite{AK1994}. \\
Set $\mathcal{T}(\lambda)=\{\text{standard tableaux of shape } \lambda \}.$ We define symbols,
$\{v(\mathbb{T})\mid \mathbb{T}\in \mathcal{T}(\lambda)\}$. Let $V(\lambda)$ is the $\C$-vector space with
these symbols as its basis.
Since our action $\tau$ on
$\cH_{r,1,n}(u_1, \ldots, u_r, q)$  and the action of $\cH_{r,1,n}(u_1, \ldots, u_r, q)$  on
 $V(\lambda)$  are  defined on generators $\{T_i\}_{i=1}^{n}$,  and $\tau$ acts trivially on
 $\{T_i\}_{i=2}^{n}$, hence we just need to study them on $T_1$.\\
 It was defined as that
 \begin{eqnarray*}
 T_1v(\mathbb{T})=\xi^{a}v_{b}v(\mathbb{T}),\quad where\quad a=a(\mathbb{T},1),\quad b=b(\mathbb{T},1),
\end{eqnarray*}
therefore we see that
 $$T_1\tau(v(\mathbb{T}))=\tau^{-1}(T_1)v(\mathbb{T})$$
 $$=\xi^{-1}\xi^{a}v_{b}v(\mathbb{T})=\xi^{a-1}v_{b}v(\mathbb{T}).$$
 Then it can be verified that (Here $i\equiv i$ ($\mod$ $p$).)
 \begin{eqnarray}
\rm{\tau}(\lambda)&=&\left(\lambda_j^{i-1}\right)\\
\rm{\tau}(\mathbb{T})&=&\left(T_j^{i-1}\right)
\end{eqnarray}
 Which implies  that $\tau$ moves all Young diagram backward for index $i$,
 but the ${\rm shift}$ moves  all Young diagram forward for index $i$,
 so $\tau \circ {\rm shift}=1$.
 %By definition of shift operator, we can see that
% $$T_1v({\rm{shift}}(\mathbb{T}))=\xi^{a+1}v_{b}v({\rm shift}(\mathbb{T})),$$
%hence our theorem holds.
\end{proof}
If there exists a $k$ such that ${\rm shift}^k(\lambda)=\mu$, we say that
$\lambda$ and $\mu$ are cyclic equivalent. We denote the cyclic equivalent class of $\lambda$ by
$\bar{\lambda}$. Let $e_\lambda=p/|\bar{\lambda}|$ and $G_\lambda$ be the stabilizer of
$\lambda$ in  $G=\left<\tau\right>=\left<{\rm shift}\right>\simeq \Z/p\Z$. We see that
$$G_\lambda=\left<\tau^{\frac{p}{e_\lambda}} \right>\simeq Z/e_\lambda\Z.$$
In \cite{A1995}, all the irreducible representation of $\cH_{r,p,n}(v_0,\ldots, v_d,q)$ are represented by
$V(\bar{\lambda}, l)$, $0\leq l\leq e_{\lambda}$. Let $\alpha_\lambda$ be function defined from
$G_{\lambda}\times G_{\lambda}$ to $\C$ and $(\C G_{\lambda})_{\alpha_{\lambda}^{-1}}$ be  defined as in
the Appendix in \cite{RR2003}.  \\
By \cite[Lemma 2.4]{A1995} and  \cite[Theorem A.6, Theorem A.13]{RR2003},
we can get the following two theorems.
\begin{thm}Let $\pi$ be a simple module (linear character)   of $(\C G_{\lambda})_{\alpha_{\lambda}^{-1}}$.
\\For the $\cH_{r,1,n}(u_1, \ldots, u_r, q)\rtimes G_{\lambda}$-module $V(\lambda)\otimes \pi$, we define that
$$RG^{V(\lambda), \pi}={\rm Ind}_{\cH_{r,1,n}(u_1, \ldots, u_r, q)\rtimes G_{\lambda}}^{\cH_{r,1,n}(u_1, \ldots, u_r, q)\rtimes G}(V(\lambda)\otimes\pi).$$
Then each simple $\cH_{r,1,n}(u_1, \ldots, u_r, q)\rtimes G$-module are obtained under this construction, and
$RG^{V(\lambda), \pi}\cong RG^{V(\lambda^{'}), \pi^{'}}$ if and only if $\lambda$ and $\lambda^{'}$ are cyclic equivalent and $\pi\cong\pi^{'}$.
 \end{thm}
 \begin{thm} Take $V(\lambda)$ as a module of $\cH_{r,1,n}(u_1, \ldots, u_r, q)\rtimes G_{\lambda}$.
 We have a decomposition as the following.
 $$V(\lambda)=\oplus_{l=1}^{e_{\lambda}}V(\bar\lambda,l)\otimes (\pi_l)^{*},$$
 as $\cH_{r,p,n}(v_0,\ldots, v_d,q)\rtimes G_{\lambda}$-{mod},
 where $\pi_l$ is a  simple $(\C G_{\lambda})_{\alpha_{\lambda}}$-module,
  and $(\pi_l)^{*}$ is the dual of the simple   $(\C G_{\lambda})_{\alpha_{\lambda}^{-1}}$-module $\pi_l$,
 and $V(\bar\lambda,l)$ is the standard irreducible module of $\cH_{r,p,n}(v_0,\ldots, v_d,q)$ in
 \cite{A1995}.
 \end{thm}
 \section{Restriction functor and KZ functor}\label{sect:res and KZ}
 \subsection{Category $\cO$ and ${\rm KZ}$ functor}
 The category $\cO$ of $H_c(W)$ is the full subcategory $\cO_c(W)$ of the category $H_c(W)$-modules consisting of objects  that are finitely generated as $\C[V]$-modules and $V$-locally nilpotent. Let's recall more conclusions about $\cO_c(W)$ from  \cite{GGOR2003}. \\
 The category $\cO_c(W)$ is a highest weight category(\cite{CPS1988}) and artinian. Let ${\rm Irr}(W)$
be the  isomorphism classes of irreducible representations of $W$. For $\sigma\in {\rm Irr}(W)$, equip it with a $\C W\ltimes \C[V^{*}]$-module structure by allowing the element in $V\in C[V^{*}]$ act by zero,
the standard module corresponding to $\sigma$ is
$$\Delta(\sigma)=H_c(W)\otimes_{\C W\ltimes \C[V^{*}]}\sigma.$$
It is an indecomposable module with a simple head $L(\sigma)$.  Both $\{L(\sigma)\}_{\sigma\in {\rm Irr}(W)}$ and $\{\Delta(\sigma)\}_{\sigma\in {\rm Irr}(W)}$ give a basis of $\C$-vector space $K(\cO_c(W))$, the Grothendick ring of $\cO_c(W)$.
We say that a module $N\in \cO_c(W)$ has a standard filtration if it admits a filtration
$$0=N_0\subset N_1\subset N_1\subset\dots\subset N_n=N$$
such that the quotient $N_i/N_{i-1}$ is isomorphic to a standard module.
 \\Following \cite{DO2003}, let $e_H=|W_H|$ and $$\epsilon_{H,j}=\frac{1}{e_H}\sum_{w\in W_H}det(w)^{j}w$$ the idempotent
 of $\C W_H$ associated to character ${\rm det}_{|W_H}^{-j}$. Given $\gamma$ as before, there is an unique
 family $\{k_{H,i}=k_{H,i}(\gamma)\}_{H\in\cA/W, 0\leq i\leq e_H}$ of elements of $k$ such that
 $k_{H,0}=k_{H,e_H}=0$ and
 \begin{eqnarray}
\gamma_H&=&e_H\sum_{j=0}^{e_H-1}(k_{H,j+1}(\gamma)-k_{H,j}(\gamma))\epsilon_{H,j}\\
&=&\sum_{w\in W_H\setminus\{1\}}\left(\sum_{j=0}^{e_H-1}{\rm det}(w)^{j}(k_{H,j+1}(\gamma)-k_{H,j}(\gamma))\right)w.\label{formulagamma}
\end{eqnarray}
For each $H\in \cA$, we put that
\begin{eqnarray}
a_H=\sum_{j=0}^{e_H-1}e_H k_{H,j}(\gamma) \epsilon_{H,j}.\label{formulaah}
\end{eqnarray}
Let $V_{{\rm reg}}=V\setminus\cup_{H\in \cA}H$ and $\mathcal{D}(V_{{\rm reg}})$ be the ring of differential operators over $V_{\rm reg}$.
In \cite{EG2002}, there is an automorphism defined from
$H_c(W)$ to $\mathcal{D}(V_{{\rm reg}})\rtimes W$ given by $x\mapsto x$, $w\mapsto w$ for
$x\in V^{*}$, $w \in W$,  and
$$y\mapsto T_{y}=\partial_{y}+\sum_{H\in A}\frac{\left<y,\alpha_H\right>}{\alpha_H} a_H$$
for $y\in V$.\\
Let $x_0\in V_{\rm reg}$, and  let $B_W=\pi_{1}(V_{\rm reg}/W, x_0)$ be the Artin Braid group
associated to $W$.
Let $H(W, V,\gamma)$ be the Hecke algebra of $W$ over $\C$ defined as the quotient of
$\C[B_W]$ by the relations (\cite{BMR1998}),
\begin{eqnarray}
(T-1)\prod_{j=1}^{e_H-1}(T-det(s)^{-j}e^{2i\pi k_{H,j}})=0\label{BMRrel}
\end{eqnarray}
for $H\in \cA$, $s\in W$ the reflection around with nontrivial eigenvalue $e^{\frac{2i\pi}{e_H}}$ and $T$
an $s$-generator of the monodromy around $H$.\\
For any $M\in \cO_c(W)$, write
$$M_{\rm reg}=M\otimes_{\C[V]}\C[V_{\rm reg}].$$ By Dunkl isomorphism,
It can be regarded as an $\mathcal{D}(V_{{\rm reg}})\rtimes W$-module which is finitely generated over
$\C[V_{\rm reg}]$. Hence $M_{\rm reg}$ is an $W$-equivariant vector bundle on $V_{\rm reg}$ with an integrable connection $\nabla$ given by $\nabla_y=\partial_y m$ for $m\in M$, $y\in V$. It is proved in \cite{GGOR2003}
that the connection $\nabla$ has regular singularities. Let $\cO_{V_{\rm reg}}^{\rm an}$ the sheaf of holomorphic function on $V_{\rm reg}$. For any
free $\C[V_{\rm reg}]$-module $N$ of finite rank, we consider $N^{\rm an}=N\otimes_{\C[V]}\cO_{V_{\rm reg}}^{\rm an}$. For $\nabla$ an integrable connection on $N$. The sheaf holomorphic horizontal sections
$$N^{\nabla}=\{n\in N^{\rm an}: \nabla_y(n)=0, \forall  y\in V\}$$
is a $W$-equivariant local system on $V_{\rm reg}$. Hence it is a local system on $V_{\rm reg}/W$.
By Riemann-Hilbert correspondence, it yields a finite dimensional representation of $\C B_W$. For $M\in\cO_c(W)$ it is proved (\cite{GGOR2003}) that the action of $\C B_W$ factors through the Hecke
algebra   $H(W, V,\gamma)$.
The Knizhinik-Zamolodchikov functor (${\rm KZ}$) is the functors,
$${\rm KZ}(W,V): \cO_c(W)\rightarrow H(W, V,\gamma),\quad M\mapsto (M_{\rm reg})^{\nabla}.$$
\subsection{Restriction from $H_c(G(r,1,n))$ to $H_c(G(r,p,n))$}
In this section, we will prove the commutativity  the restriction functor  from  $\cH_{r,1,n}(u_1, \ldots, u_r, q)$ to
$\cH_{r,p,n}(v_0,\ldots, v_d,q)$ and the ${\rm KZ}$ functors.\\
In order to  make the definition from \cite{BMR1998} and the definition from \cite{RR2003} compatible and have the Hecke algebra of type
$G(r,p,n)$ as a subalgebra of type $G(r,1,p)$, we have the relation below:
\begin{eqnarray}\label{i=jmodd}
e^{2i\pi k_{H,i}}=e^{2i\pi k_{H,j}}\quad {\rm if }\quad i\equiv j \quad \mod \quad d,
\end{eqnarray}
for $H$ in the  $\mathcal{C}_0$, $G(r,1,n)$-orbit of the hyperplane corresponding to $S_1$.
This will make the relation (\ref{t1uvrel}) and (\ref{BMRrel}) coincide.  \\
Let ${\rm refl}(r,1,n)$ and ${\rm refl}(r,p,n)$ are reflections of $G(r,1,p)$ and $G(r,p,n)$, respectively.
We know that $${\rm refl}(r,1,n)\setminus{\rm refl}(r,p,n)=\{w\in W_H\mid
{\rm det}(w)^{\frac{r}{p}}\neq 1,H\in\mathcal{C}_0 \}\subset\cup_{H\in\mathcal{C}_0} W_H.$$
   The formula (\ref{formulaah}) for $a_H$  can be rewritten as the following
$$a_H=\sum_{w\in W_H}\left(\sum_{j=0}^{e_H-1} k_{H,j}{\rm det}(w)^{j}\right)w.$$
Here  we define the coefficient of $w\in {\rm refl}(r,1,n)\setminus{\rm refl}(r,p,n)$ to be $0$, namely
\begin{eqnarray}\label{khiwj}
\sum_{j=0}^{e_H-1} k_{H,j}{\rm det}(w)^{j}=0,
\end{eqnarray}
then we can verify that the corresponding coefficient in
(\ref{formulagamma}) for $\gamma_H$ is $0$,
%under the supposition that
%\begin{eqnarray}\label{khieqal}
%k_{H,1}=0
%\end{eqnarray}
because $k_{H,0}=k_{H,e_H}=0$ and
$$\sum_{j=0}^{e_H-1}{\rm det}(w)^{j}k_{H,j+1}={\rm det}(w)^{-1}\sum_{j=1}^{e_H}{\rm det}(w)^{j}k_{H,j}=0.$$
\begin{rem} We can replace (\ref{i=jmodd}) by  a stronger equation as below
\begin{eqnarray}\label{i=jmodd2}
 k_{H,i}= k_{H,j}\quad {\rm if }\quad i\equiv j \quad \mod \quad d,
\end{eqnarray} and we can write
\begin{eqnarray*}
a_H&=&\sum_{w\in W_H}\left(\sum_{j=0}^{e_H-1} k_{H,j}{\rm det}(w)^{j}\right)w\\
   &=& \sum_{j=0}^{e_H-1} k_{H,j}+\sum_{w\in W_H\cap {\rm refl}(r,p,n)}\left(\sum_{j=0}^{e_H-1} k_{H,j}{\rm det}(w)^{j}\right)w\\
   &+&\sum_{w\in W_H\cap({\rm refl}(r,1,n)\setminus{\rm refl}(r,p,n))}\left(\sum_{j=0}^{e_H-1} k_{H,j}{\rm det}(w)^{j}\right)w.
\end{eqnarray*}
 For $w\in W_H\cap({\rm refl}(r,1,n)\setminus{\rm refl}(r,p,n))$, we have ${\rm det}(w)^d\neq 1$, therefore
 it follows that
 \begin{eqnarray*}
 \sum_{j=0}^{e_H-1} k_{H,j}{\rm det}(w)^{j}
 &=&\sum_{l=0}^{d}\sum_{t=0}^{p-1}{\rm det}(w)^{td+l}k_{H,td+l }\\
 &=&\sum_{l=0}^{d}\left(\sum_{t=0}^{p-1}{\rm det}(w)^{td+l}\right)k_{H,l }=0
\end{eqnarray*}
which implies that $a_H$ vanishes on  $w\in W_H\cap({\rm refl}(r,1,n)\setminus{\rm refl}(r,p,n))$.
\end{rem}
\begin{thm}\label{KZREscommute}
Suppose that the parameters of $H_c(G(r,1,n))$ satisfy relations (\ref{i=jmodd}) and (\ref{khiwj}).
Let $Res_{r,p,n}$ and  $Res_{r,p,n}^{'}$ be the restriction functor from   $\cH_{r,1,n}(u_1, \ldots, u_r, q)$-mod to $\cH_{r,p,n}(v_0,\ldots, v_d,q)$-mod and $\cO_{c}(G(r,1,n))$ to $\cO_{c}(G(r,p,n))$, respectively.
Then we have that the following diagram commutes.
\begin{center}
\xymatrix{
\cO_{c}(G(r,1,n))\quad\ar[r]^{Res_{r,p,n}^{'}}\ar[d]^{{\rm KZ}(G(r,1,n),V)}&\quad\cO_{c}(G(r,p,n))\ar[d]^{{\rm KZ}(G(r,p,n),V)}\\
\cH_{r,1,n}(u_1, \ldots, u_r, q)-{\rm mod}\ar[r]^{Res_{r,p,n}}&\cH_{r,p,n}(v_0,\ldots, v_d,q)-{\rm mod}
}
\end{center}
\end{thm}
\begin{proof} Let $\cA(r,1,n)$ and $\cA(r,p,n)$ be the set of reflection hyperplanes  of $G(r,1,n)$
and $G(r,p,n)$, respectively.
We write $V_{\rm reg}=V\setminus\cup_{H\in \cA(r,1,n)}H$ and  $V_{\rm reg}^{'}=V\setminus\cup_{H\in \cA(r,p,n)}H$. Let $\phi_{Dl}$ and $\phi_{Dl}^{'}$ be the Dunkl isomorphisms defined on $H_{c}(G(r,1,n))$
and $H_C(G(r,p,n))$. By assumption, we have
$$\phi_{Dl}(y)=\phi_{Dl}^{'}(y)=\partial_y+\sum_{H\in A}\frac{\left<y,\alpha_H\right>}{\alpha_H} a_H.$$
Therefore the $\nabla_y$ for $H_{c}(G(r,1,n))$
and $H_C(G(r,p,n))$ are the same.
And we see that the quotient field of $V_{\rm reg}$ and $V_{\rm reg}^{'}$ are the same.
Therefore it can be verified that
$$Res_{r,p,n}{\rm KZ}(G(r,p,n),V)(M)=Res_{r,p,n}((M\otimes\C[V_{\rm reg}])^{\nabla}))$$
$$=(Res_{r,p,n}^{'}(M)\otimes\C[V_{\rm reg}]^{'})^{\nabla}={\rm KZ}(Res_{r,p,n}^{'}(M)),$$
for $M\in\cO_{c}(G(r,1,n))$.
 \end{proof}
 %\begin{rem} Only when $p=r$, it follows that  $V_{\rm reg}\neq V_{\rm reg}^{'}$.
 %\end{rem}
\begin{cor} We have the following result for $C[V]\otimes V(\lambda)$,
$$Res_{r,p,n}^{'}(C[V]\otimes V(\lambda))=\sum_{l=1}^{e_\lambda}C[V]\otimes V(\bar{\lambda},l).$$
\end{cor}
If we consider the $\tau$-action on the $\cO_{c}(G(r,1,n))$ induced by the automorphism
through the  $\tau$-action on $H_c(G(r,1,n))$, we can get the following Theorem.
%This will use some trick  in \cite[Lemma 2.1]{S2011} by the propositions
%that $\cO_c(G(r,1,n))$ has enough projective objects and any projective
%objects has a standard filtration.
\begin{thm} The $\tau$-action on $\cO_{c}(G(r,1,n))$ and on $\cH_{r,1,n}(u_1, \ldots, u_r, q)$-mod commutes with ${\rm KZ}(G(r,1,p), V)$, namely,
$${\rm KZ}(G(r,1,p), V)\circ \tau=\tau \circ {\rm KZ}(G(r,1,p), V).$$
\end{thm}
\begin{proof}
Since $\tau$ acts on the $\C[V]$ and $\C[V^*]$ trivially, and on the generators of $G(r,1,n)$ just by changing
$S_1$ to $\xi S_1$, then by the monodromy, it will naturally have the same action on the corresponding
Hecke algebra on the  corresponding generators. At the same time that there is an projective module $P_{KZ}$
in $\cO_c(G(r,1,n))$ such that the $KZ$ functor is isomorphic to ${\rm Hom}_{\cO_c(G(r,1,n))}(P_{KZ},-)$.
Therefore the theorem follows from
$${\rm Hom}_{\cO_c(G(r,1,n))}(P_{KZ},M^{\tau})\cong {\rm Hom}_{\cO_c(G(r,1,n))}(P_{KZ},M)^{\tau},$$
for any $M\in \cO_c(G(r,1,n))$.
\end{proof}
In fact, If we consider the Heckeman-Opdam shift operator in \cite{BC2011},
the automorphism $\tau$ can be represented as an Heckeman-Opdam shift operator on the parameters $\{k_{H,i}\}$ as the following,
\begin{eqnarray*}
\tau(k_{H,i})&=&k_{H,i}-\frac{1}{p}\quad for \quad H\in C_0,\\
 \tau(k_{H,i})&=&k_{H,i}\quad otherwise.
\end{eqnarray*}
 Let $\chi$ be the character function on  $G(r,1,p)$ defined by
 $\chi(S_1)=\xi^{-1}$ and $\chi(S_i)=1$ for $1\leq i\leq 2$.
 Hence we see that $\tau(V(\lambda))\cong \chi\otimes V(\lambda)$.
 By \cite[Theorem 4.2]{O1998}, we have the following corollary.
 \begin{cor} Let $F_{V(\lambda)}(t)$ be the fake degree of $V(\lambda)$ in \cite{O1998}.
 Then there exists an $m\in \Z$, such that
 $$F_{\tau(V(\lambda))}(t)=t^{m}F_{V(\lambda)}(t).$$
 Therefore, for the $r$-tuple Young diagrams $\lambda$, $\lambda'$ are on the same orbit of Ariki's $shift$ operator, $F_{V(\lambda')}(t)= t^{m'}F_{V(\lambda)}(t)$ for some $m'\in \Z$.
 \end{cor}
 \subsection{From type $\ddB_n$ to type $\ddD_n$}
 For the classical  Weyl groups, obtaining the type $\ddD_n$ from type $\ddB_n$ is the most canonical application of Clifford theory or spin theory. It is known that the Weyl groups $W(\ddB_n)$ and $W(\ddD_n)$
 are isomorphic to $G(2,1,n)$ and $G(2,2,n)$, respectively. The irreducible representations of $W(\ddB_n)$
 are one to one corresponding to the bipartition of $n$, namely
 $$\{V(\lambda)\mid \lambda=(\lambda_1,\lambda_2),\quad |\lambda_1|+|\lambda_2|=n\}.$$
 %For the representation of $W(\ddD_n)$, the irreducible representations are
 %$$\{\bar{\lambda}=(\lambda_1,\lambda_2)=(\lambda_2,\lambda_1)\mid \lambda_1\neq \lambda_2\}$$
% and $$\{\bar{\lambda,0}=(\lambda_1,\lambda_2)=\mid \lambda_1=\lambda_2\}$$
%those  bipartition $\lambda$ labeled by $$
By the restrction from  $W(\ddB_n)$ to $W(\ddD_n)$, we get the decomposition of bipartition module:\\
$${\rm Res}_{\ddD_n}^{\ddB_n}(V((\lambda_1,\lambda_2)))\cong{\rm Res}_{\ddD_n}^{\ddB_n}(V((\lambda_2,\lambda_1))),\quad  \lambda_1\neq \lambda_2,$$
$${\rm Res}_{\ddD_n}^{\ddB_n}(V((\lambda_1,\lambda_2)))\cong V((\lambda_1,\lambda_2),0)\oplus
V((\lambda_1,\lambda_2),1),\quad  \lambda_1=\lambda_2.$$
which are all the irreducible modules of $W(\ddD_n)$.\\
The reflection representation of $W(\ddB_n)$
and $W(\ddD_n)$  is realized in $\C^{n}$. The reflection hyplanes for
$W(\ddB_n)$
and $W(\ddD_n)$  are  defined by linear equations $\{x_i,\, x_i\pm x_j\}_{1\leq i\neq j\leq n}$
and $\{x_i\pm x_j\}_{1\leq i\neq j\leq n}$, respectively.
If all parameters for $H_c(W(\ddB_n))$ satisfy the relations in the last section,
then we get the decomposition of the standard module for  $H_c(W(\ddB_n))$ for restrcition
to $H_c(W(\ddD_n))$.
\begin{eqnarray*}
\C[\C^n]\otimes V(\lambda)&\cong& \C[\C^n]\otimes V(\bar{\lambda}), \,  \lambda=(\lambda_1,\lambda_2),\lambda_1\neq \lambda_2,\\
\C[\C^n]\otimes V(\lambda)&\cong& \C[\C^n]\otimes V(\bar{\lambda},0)\oplus \C[\C^n]\otimes V(\bar{\lambda},1),
\lambda=(\lambda_1,\lambda_2),\lambda_1=\lambda_2.
\end{eqnarray*}
{\bf Acknowledgement} At the end of the paper, I appreciate   Prof E.~Opdam for his supervision for my postdoc study in UvA and many helpful talks
to produce this paper.

Shoumin Liu\\
Email: s.liu@sdu.edu.cn\\
Taishan College, Shandong University\\
Shanda Nanlu 27, Jinan, \\
Shandong Province, China\\
Postcode: 250100


\begin{thebibliography}{9}
\bibitem{AK1994}
S. Ariki, K. Koike, A Hecke algebra of $(\Z/r\Z)\wr S_n$ and construction of its irreducible
representation , Adv. in Math. {\bf 106}(1994),216--243.
\bibitem{A1995}
S. Ariki, Representation theory of a Hecke algebra of $G(r,p,n)$, J. Algebra {\bf 177}(1995), 164--185.
\bibitem{DO2003}C. Dunkl, E. Opdam, Dunkl operators for complex reflection groups, Proc. London Math. Soc.
{\bf 86} (2003), 70--108.
\bibitem{BC2011}
Y. Berest, O. Chalykh,
Quasi-invariants of complex reflection groups, Compositio Math. {\bf 147}(2011), 965--1002.
\bibitem{BMR1998}
M.Brou\'e, G. Malle and R. Rouquier, Complex reflection groups, braid groups, Hecke algebras,  J. reline angew. Math.
{\bf 500} (1998), 127-190.
\bibitem{CPS1988}
E.Cline, B.Parshall and L.Scott, Finite dimensional algebras and highest weight categories, J. reline angew. Math.
{\bf 391} (1988), 85-99.
\bibitem{EG2002}
P.Etingof, V. Ginzburg, sympletic reflection algebras, Calogero-Moser space, and deformed Harish-Chandra homomorphism, Invent. Math.
{\bf 147}(2002), 243--348.
\bibitem{GGOR2003}
V. Ginzburg, N. Guay, E. Opdam, R. Rouquier, On the category $\mathcal{O}$ for rational Cherednik algebra,
Invent. Math.
{\bf 154}(2003), 617--651.
\bibitem{L1989} G. Lusztig, Affine hecke algebras and their graded version,
J.Amer. Math.Soc.{\bf 2}(1989), 599--635.
\bibitem{O1998}
E. Opdam, Complex reflection groups and fake degrees, \url{arxiv:math/9808026}.
\bibitem{RR2003}
A. Ram, J. Ramagge,
Affine Hecke algebras, cyclotomic Hecke algebras and Clifford theory, A tribute to C. S. Seshadri (Chennai, 2002), 428–-466,
Trends Math., Birkhäuser, Basel, 2003.
\bibitem{S2011} Peng Shan, Crystals of Fock spaces and cyclotomic rational double affine Hecke algebras,
 Ann. Sci. \'{E}c. Norm. Sup\'{e}r.  {\bf 44}(2011), 147--182.
\end{thebibliography}
\end{document}